\documentclass{amsart}
\usepackage{graphicx}
\usepackage{cite}
\usepackage[hidelinks]{hyperref}
\usepackage[nameinlink, capitalise]{cleveref} 
\crefname{equation}{}{}  
\crefname{figure}{Figure}{Figures}
\usepackage{subcaption}

\usepackage{tikz}
\usepackage{amsthm}
\usepackage{amssymb}
\usetikzlibrary{decorations.pathreplacing,angles,quotes}

\title{Sizes of witnesses in Covtree}

\author{Jette Gutzeit${}^{a}$}
\address{${}^{a}$Department of Mathematics, University of Zurich,
  Zurich, Switzerland}
\email{jette.gutzeit@icloud.com}

\author{Kimia Shaban${}^{b}$}
\address{${}^{b}$Department of Computer Science, University of Toronto,
  Toronto, ON, Canada}
\email{kimia.shaban@mail.utoronto.ca}

\author{Karen Yeats${}^{c}$}
\address{${}^{c}$Department of Combinatorics and Optimization, University of Waterloo,
  Waterloo, ON, Canada}
\email{kayeats@uwaterloo.ca}

\author{Stav Zalel${}^{d}$}
\address{${}^{d}$Department of Applied Mathematics and Theoretical Physics,
University of Cambridge, CB3 0WA, United Kingdom and
Homerton College, University of Cambridge, CB2 8PQ, United Kingdom}
\email{sz435@cam.ac.uk}

\thanks{KS and KY thank Jonathan Leake for his insights on and name for the exchange graph of downsets.  KY also thanks Zeus Dantas E Moura for catching an error in an earlier version.  KY is supported by an NSERC Discovery grant and by the Canada Research Chairs program.  KY thanks the Perimeter Institute for hosting the beginning of this work and thanks the hospitality of Homerton College, Cambridge. 
Research at Perimeter Institute is supported in part by the Government of Canada through the Department of Innovation, Science and Economic Development Canada and by the province of Ontario through the Ministry of Economic Development, Job Creation and Trade. This research was also supported in part by the Simons Foundation through the Simons Foundation Emmy Noether Fellows Program at Perimeter Institute.}

\tikzset{
    D/.style={draw, circle, inner sep=0pt,
    minimum size=4pt, fill},
    B/.style={draw=none, circle, inner sep=0pt,
    minimum size=4pt, fill=cyan},
    R/.style={draw=none, circle, inner sep=0pt,
    minimum size=4pt, fill=red}
}

\newtheorem{theorem}{Theorem}[section]
\theoremstyle{definition}
\newtheorem{definition}[theorem]{Definition}
\newtheorem{lemma}[theorem]{Lemma}
\newtheorem{proposition}[theorem]{Proposition}

\theoremstyle{remark}
\newtheorem{remark}[theorem]{Remark}

\begin{document}
\begin{abstract}
Given a set $\Gamma$ of $k$ unlabelled posets, each of size $n$, we say that a poset $Q$ is a \emph{witness} to $\Gamma$ if $\Gamma$ is the set of downsets of size $n$ of $Q$. We say that $Q$ is a \emph{minimal witness} if it does not contain a proper downset that is itself a witness to $\Gamma$. Motivated by the causal set approach to quantum gravity, we study the upper bound on the size of minimal witnesses as a function of $n$ and $k$. We show that there is no linear upper bound of the form $n+k+c$ for any constant $c$. We introduce the \emph{exchange graph of downsets} as a new tool to study this scenario, and use it to show that all minimal witnesses $Q$ satisfy the bound $|Q|\leq nk-n$, and that when $k=3$ there is at least one minimal witness $Q$ that satisfies the bound $|Q|\leq \frac{3}{2}(n+1)$. 
\end{abstract}
\maketitle

\section{Introduction}

Causal set theory is an approach to quantum gravity in which spacetime is fundamentally discrete and dynamical \cite{Bombelli:1987aa,Surya:2019ndm,Dowker2024}. This approach stands out for its simple and physically-motivated axioms and for its intrinsic Lorentzianness. It is also very appealing mathematically since it asks interesting questions about locally finite posets that one would not be led to ask otherwise.

In causal set theory, the spacetime manifold-metric pair of the continuum is replaced by a \emph{causal set}, a locally finite poset where the partial order encodes the lightcone structure \cite{Bombelli:1989mu}. In this scenario, the quantum dynamics of spacetime would take the form of a path integral---or rather a path sum---over weighted posets, and the aim of much research in this area is to find physically meaningful weights. One major approach has been to consider physically-motivated probability measures on classes of posets \cite{Rideout:1999ub,Ash:2002un,Ash:2005za,Dowker:2005gj,Brightwell2016,Brightwell:2011,Brightwell:2012,Brightwell:2002vw,Bento:2021voo}. These may be considered as classical dynamics whose quantisation is then sought \cite{Dowker:2010qh,Surya:2020cfm,Dowker:2022}.

Two physically-motivated principles that have been used to define probability measures are relevant for us in this work: growth and label-independence. Growth is the notion that a poset comes into being by the accretion of elements. In causal set theory, it is hoped that this notion of growth can encode the passage of time---a physical experience that is yet to be described in physics \cite{Dowker:2020qqs,Dowker:2014xga}. Growing posets have also been extensively studied in other contexts within the combinatorics literature, and two of the authors of this work have used this to connect causal set growth dynamics with the Hopf algebra literature \cite{Yeats_Zalel_2024}. Label-independence is the discrete analogue of general covariance---the fundamental assumption underlying General Relativity that physics is independent of coordinates. The idea is that spacetime is actually an \emph{unlabelled} poset. While we work with labelled posets for convenience, our physical conclusions should not depend on the labelling that we choose.  

Growth and label-independence are in some sense at odds with each other \cite{Brightwell:2002yu}. One way to phrase the condition of label-independence is that it is only the partial order, and no total order, that matters to physics. But the natural notion of growth necessarily introduces a total order as the elements are born into the poset one by one. In some ways this dissonance has been fruitful, giving rise to physically-motivated dynamics by using label-independence as a constraint \cite{Rideout:1999ub,Brightwell:2002vw,Dowker:2022}. But this dissonance leaves more to be desired: can one formulate a growth process for unlabelled posets, where labels and total orders never appear at all?

This is the motivation for \emph{covtree} (short for covariant tree), a tree whose nodes are certain sets of unlabelled posets \cite{Dowker:2019qiz,Zalel:2020oyf,Zalel2024}. Random walks up covtree yield probability measures on the space of unlabelled posets without labels ever entering the game. The unlabelled poset grows with each step up the tree. This growth does not correspond to any given element being born but to a more global notion of growth that is somewhat less intuitive to pin down \cite{Wuthrich:2015vva}.

Conceived within causal set theory, covtree is an interesting combinatorial object in its own right. In particular, the fact that it is simple to define but difficult to explicitly construct leaves much for exploration. Inspired by this, here we are interested in the following purely combinatorial question about posets.  Let $P_1, \ldots, P_k$ be distinct unlabelled posets, each with $n$ elements.  Suppose there is a poset $Q$ such that $\Gamma_n=\{P_1, \ldots, P_k\}$ is exactly the set of downsets of $Q$ of size $n$ up to isomorphism.  We call such a $Q$ a \emph{witness} for $\Gamma_n$.  If any such witness $Q$ exists then there will be infinitely many, since adding a new maximal element above all elements of $Q$ will create a larger poset with the same downsets of size $n$.  One can ask what the minimum size of $Q$ is for a given $\Gamma_n$, and perhaps most interestingly, how large can this minimum be for fixed $n$ and $k$?  Put slightly differently, given $n$ and $k$, can one give a tight bound as a function of $n$ and $k$, so that if $\Gamma_n$ has a witness then it has one of size at most this bound?

Here, we are not able to answer this question in full, but give some partial results and families of examples. We hope by this work to inspire others to consider this and other combinatorial puzzles posed by covtree.

\section{Technical background and preliminary observations}\label{sec background}

\subsection*{Basic terminology}
To briefly set our poset notation, a \emph{poset} $P$ is a set (by standard abuse of notation also called $P$), with a reflexive, antisymmetric, and
transitive relation which we write $\leq$.  For $a, b\in P$ we write $a<b$ if $a\leq b$ and $a\neq b$.  We say $b$ \emph{covers} $a$ if $a<b$ and there is no $c\in P$ such that $a<c<b$.  We say that $a$ and $b$ are \emph{comparable} if either $a\leq b$ or $b\leq a$, otherwise they are \emph{incomparable}. 
A \emph{downset} in a poset $P$ is a subset $D$ of $P$ that is downwards closed, that is for any $a\in D$ and $b\in P$, if $b\leq a$ then $b\in D$.
A \emph{chain} $C$ is a totally ordered subposet of $P$, that is every pair of elements of $C$ is comparable; the \emph{length} of a chain is the number of elements in the chain.  An \emph{antichain} is a subposet of $P$ where every pair of elements is incomparable. Given a poset $P$ and an element $x\in P$, the \emph{height} of $x$, $h(x)$, is the length of the longest chain in $P$ that has $x$ as its maximal element. The \emph{height} of $P$ is the length of the longest chain in $P$ or equivalently the height of the tallest element in $P$.

\subsection*{Sequential growth models} Of the growth models of causal set theory, the model known as transitive percolation is the simplest and most classical. At each step, a new element is added by choosing a subset of size $a$ of the $n$ existing elements with probability $p^aq^{n-a}$ for some $0<p<1$ and $q=1-p$, adding a new element above it and taking the transitive closure \cite{Rideout:1999ub}. Taking a slightly different perspective, this is the same as what are known as random graph orders \cite{AFrgo}.  Importantly, transitive percolation is a process that builds labelled posets, with each new element getting the next label, as are the Classical Sequential Growth models that generalize it \cite{Rideout:1999ub}.

To obtain a growth model that builds unlabelled posets, naively one could take the tree of possibilities of transitive percolation (whose nodes are finite labelled posets, as shown in figure \ref{labelled poscau}) and identify the nodes up to isomorphism. This gives the poset of unlabelled posets under containment as a downset, shown in figure \ref{poscau}.  However, this is very much no longer a tree and so is quite different to work with, especially in the context of extending the the transition probabilities of a random walk into a measure over the space of posets.
\begin{figure*}[h]
    \centering
    \begin{subfigure}[b]{0.5\textwidth}
        \centering
        \includegraphics[height=1.2in]{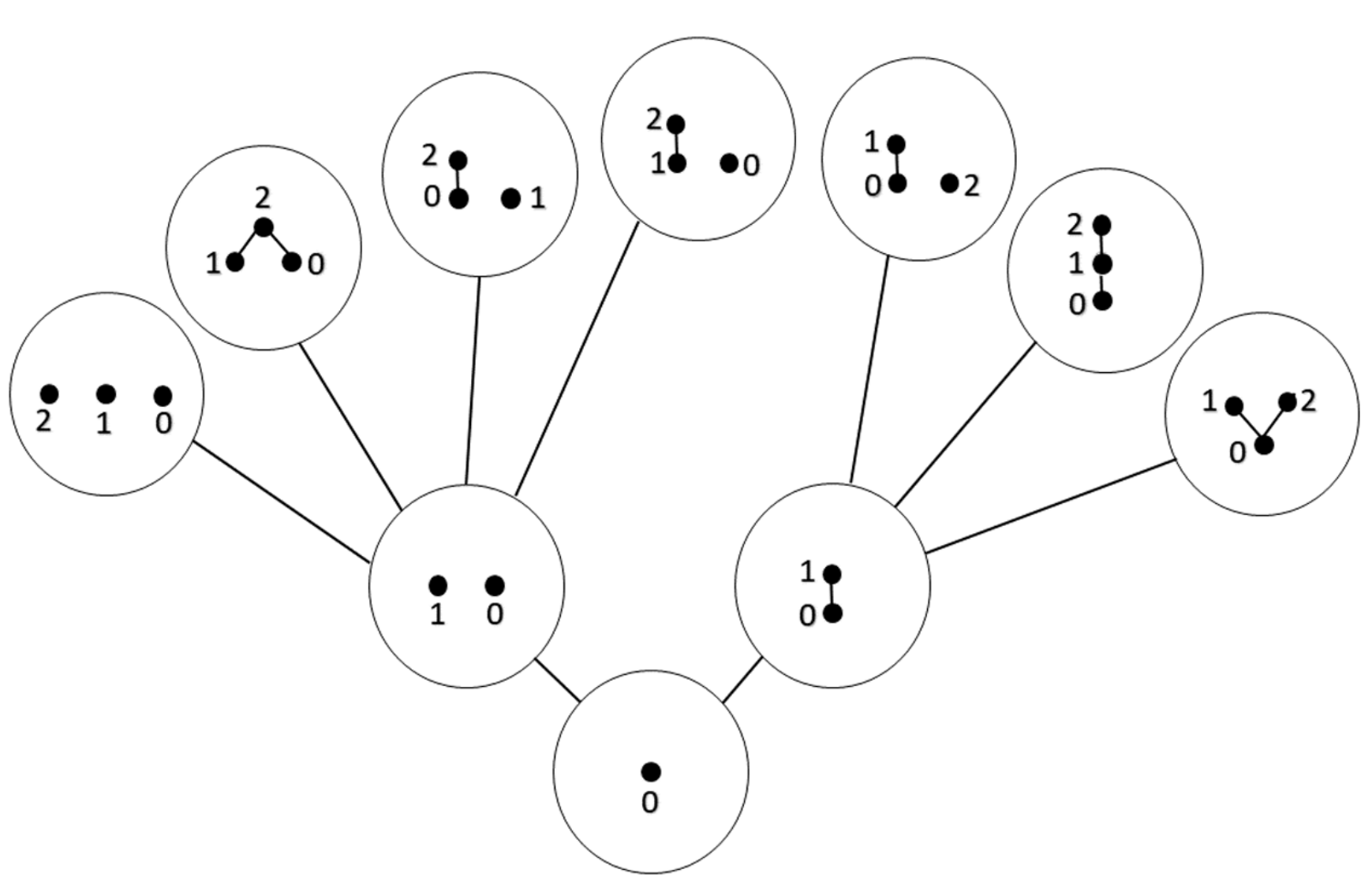}
        \caption{The tree of labelled posets.}\label{labelled poscau}
    \end{subfigure}%
    ~ 
    \begin{subfigure}[b]{0.5\textwidth}
        \centering
        \includegraphics[height=1.2in]{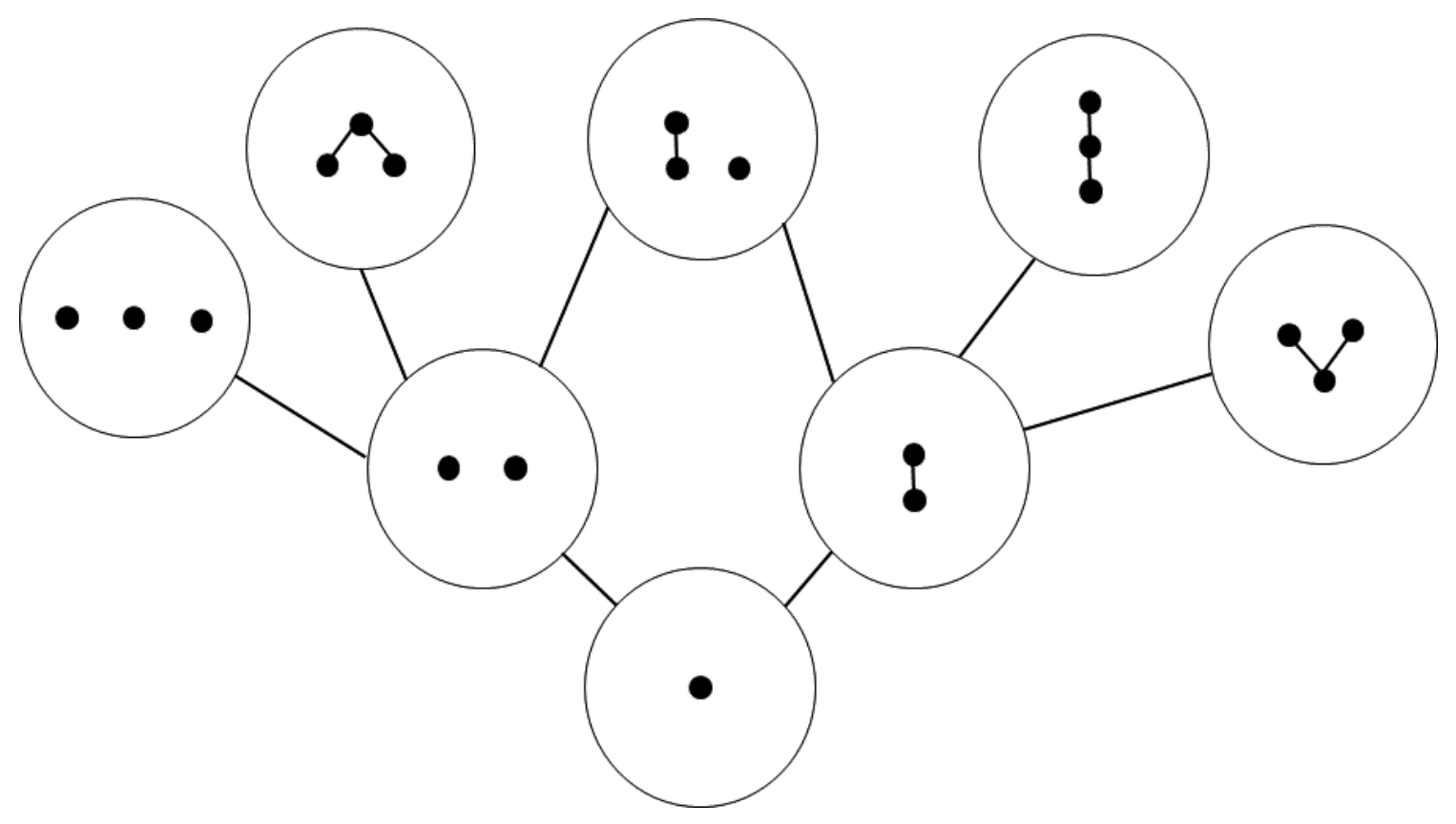}
        \caption{The poset of unlabelled posets.}\label{poscau}
    \end{subfigure}
    \caption{The tree of labelled possibilities of sequential growth and the result once the nodes are identified up to isomorphism.}
\end{figure*}
\subsection*{Covtree} One of us, with other coauthors, introduced covtree \cite{Dowker:2019qiz} in order to have a growth model that is covariant at every step while retaining a tree structure. The nodes at level $n$ of covtree are sets of the form $\Gamma_n=\{P_1, \ldots, P_k\}$, where each of the $P_i$ is an unlabelled poset of size $n$. A set $\Gamma_n$ is a node in covtree if and only if there exists some poset $Q$ of which $\Gamma_n$ is precisely the set of downsets of size $n$ up to isomorphism. Thus $Q$ witnesses that $\Gamma_n$ is a node of covtree and we say that $Q$ is a \emph{witness}\footnote{In \cite{Dowker:2019qiz}, $Q$ was called a \emph{certificate} of $\Gamma_n$.} of $\Gamma_n$. Physically, a node $\Gamma_n$ corresponds to the statement: \emph{the unlabelled downsets of the growing posets are $P_1, \ldots, P_k$.}

It is worth noting that not all sets $\Gamma_n$ are nodes of covtree, because some such sets have no witness.  For instance, every candidate witness of the set, $$\Gamma_3=\{\bullet \bullet \bullet\ ,  \includegraphics{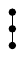}\}$$ necessarily contains $\bullet \includegraphics{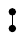}$ as a donwset of size 3, and is therefore not a witness of $\Gamma_3$.

The nodes are arranged in covtree by the following rule. A node $\Gamma_n$ is a parent of another node $\Gamma_m$ if and only if $\Gamma_n$ is the set of downsets of the $P_i\in\Gamma_m$ and $n=m-1$. In other words, the parent of $\Gamma_m$ is the set of posets that can be generated by picking an $P_i\in\Gamma_m$ and removing one of its maximal elements. For illustration, the first three levels of covtree are shown in figure \ref{covtree}.

\begin{figure}[h]
    \includegraphics[width=0.75\textwidth]{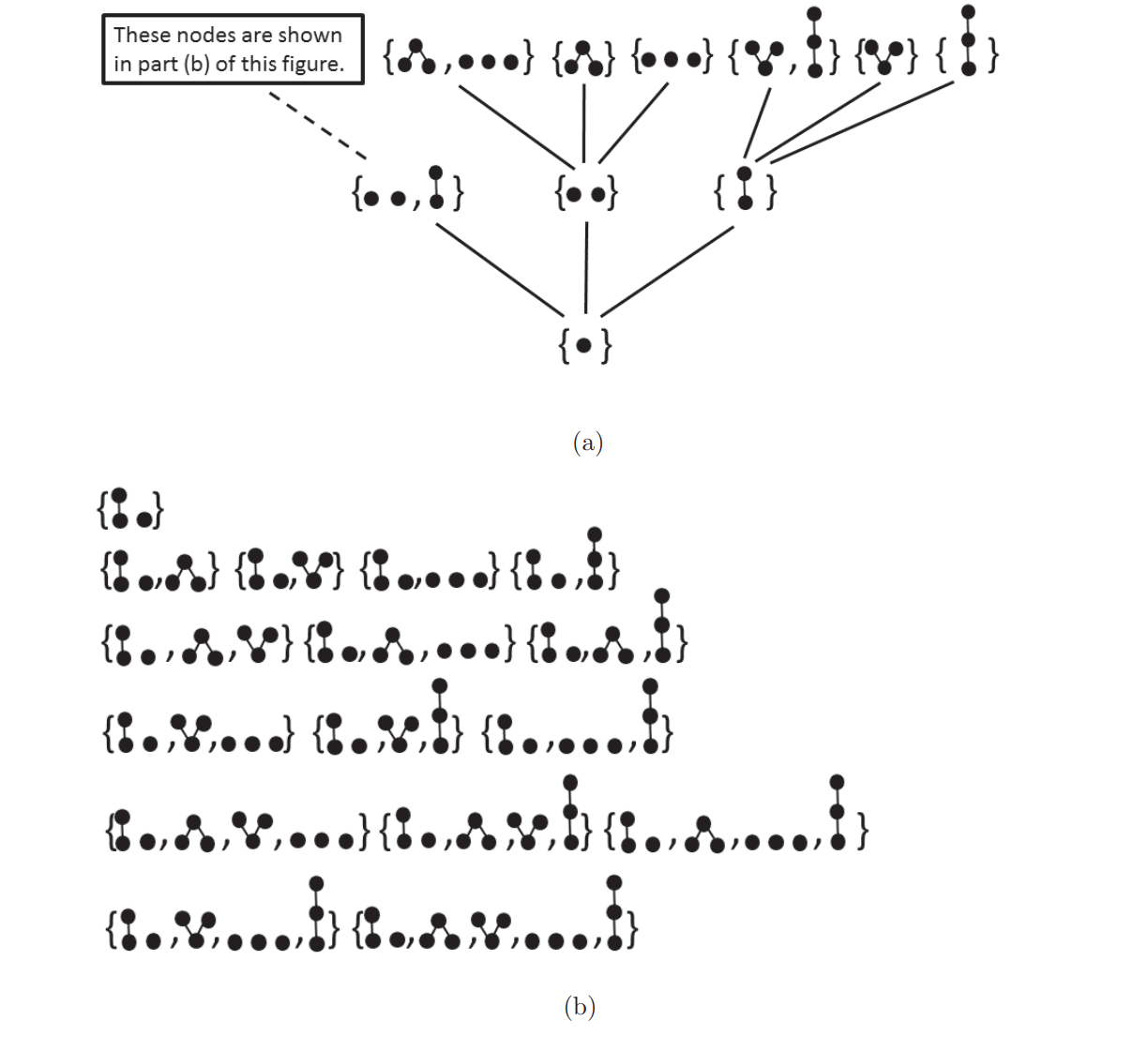}
    \caption{The first three levels of covtree. Taken from \cite{Dowker:2019qiz}.}\label{covtree}
\end{figure}

Given a node, it is very easy to construct its path down to the root of covtree by the above algorithm. The converse is far from true: in general, it is very difficult to generate the nodes above a given node. It is not easy to recognize when a set of posets is a node of covtree and there is no known growth rule to generate the children of a node. Doing so by brute force rapidly gets out of hand as the number of posets of cardinality $n$ increases rapidly with $n$. This is the motivation for our central question.

The notion of minimality of a witness will be particularly important for us. The set of witnesses to a given node of covtree forms a poset by containment as a downset. We say that $Q$ is a \emph{minimal witness} if it is minimal in this poset of witnesses. In other words, $Q$ is a minimal witness to $\Gamma_n$ if it does not contain a proper downset that is itself a witness to $\Gamma_n$. A node may have arbitrarily many minimal witnesses (cf. Properties 9 and 11 in \cite{Zalel:2020oyf}), and its minimal witnesses may be of different sizes. The interplay between the minimality of a witness and its size is at the heart of this work. Our central question can be phrased as: what are the smallest and biggest minimal witnesses that a node $\Gamma_n$ can have as a function of $n$ and $k=|\Gamma_n|$? In section \ref{section n+k+c}, we consider nodes that each have a unique minimal witness---in this special case, the minimal witness is also the smallest witness in size---and use them to show that the size of the smallest witness cannot be bounded from above by the form $n+k+c$ for any constant $c$. In section \ref{subsec exchange graph}, we give an upper bound on size that is satisfied by all minimal witnesses of a given node. And in section \ref{subsec k=3}, we give an upper bound for the size of the smallest minimal witness in the case where $\Gamma_n$ contains exactly 3 posets.

Two important results in the direction of our central question are known from previous work of one of us.  First, if a witness exists then some witness of size at most $nk$ exists.  This is Lemma 3.1 of \cite{Dowker:2019qiz}.  The proof is both simple and insightful. It will be helpful for us to reproduce it here. Suppose a witness $Q$ exists.  Within $Q$ take one isomorphic copy of each $P_i$.  The union of these within $Q$ has size at most $nk$ and is a downset of $Q$.  Furthermore, it has no downsets not among the $P_i$ for if it did then so would $Q$.  Therefore this union is itself a witness of size at most $nk$.  This bound, however, is far from tight and improving on it has been the main motivation behind this work.

The second result of importance to us is that an exact bound is known in the special cases when $k=1$ and $k=2$. In the first case, the bound above is trivially true and exact, since every witness of $\Gamma_n=\{P_1\}$ must be of size at least $n$ in order to contain $P_1$ as a downset. Indeed, in this case $P_1$ is itself the witness of $\Gamma_n$.

A more interesting scenario arises when $k=2$. If $\Gamma_n=\{P_1,P_2\}$ then every minimal witness of it must be of size $n+1$. This is Property 11 in \cite{Zalel:2020oyf}. The essence of this proof comes from the fact that when $k=2$, $\Gamma_n$ has a witness only if $P_2$ can be formed from $P_1$ by removing one maximal element and replacing it with a new maximal element.

One might hope to extend this argument to higher  values of $k$. The idea would be that one can form a path from all the $P_i\in\Gamma_n$ where each $P_i$ is constructed from the one that precedes it in the path by the replacement of a single maximal element. Without loss of generality, let this path be $P_1-P_2-\dots-P_k$. The naive expectation is then that if $\Gamma_n$ has a witness then it must have a witness of at most size $n+k-1$.  While the general idea is still helpful, this naive bound does not hold in general because spurious downsets can appear in the candidate witness $Q$ if some $P_i$ and another $P_j$ appear too close together in $Q$. In those cases, a witness may still exist but its size would be larger than $n+k-1$ in order for the downsets to be far enough from each other to not give rise to the spurious downsets. This corresponds to a path where at least one $P_i$ appears more than once.

A specific example where this happens is $\Gamma_5=\{
\includegraphics[width=10mm,scale=0.5]{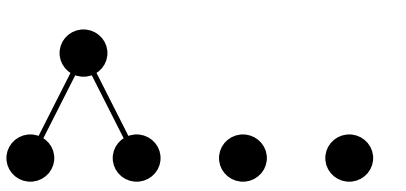},\includegraphics[width=8mm,scale=0.5]{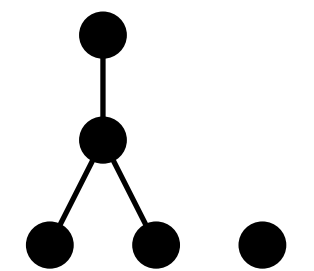},\includegraphics[width=8mm,scale=0.5]{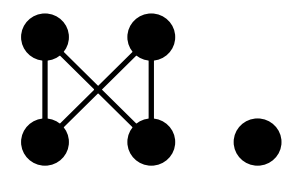}\}.$ One path that one could construct is $\includegraphics[width=10mm,scale=0.5]{C1}-\includegraphics[width=8mm,scale=0.5]{C2}-\includegraphics[width=8mm,scale=0.5]{C3}$. This path gives the false witness \includegraphics[width=10mm,scale=0.5]{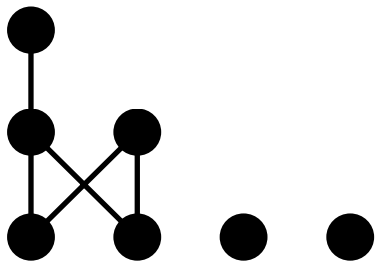} with the spurious downset \includegraphics[width=5mm,scale=0.5]{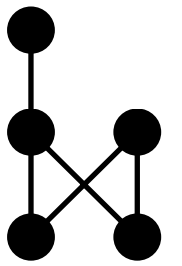}. One can repeat this exercise with the remaining permutations of the three posets to see that this $\Gamma_5$ has no witness of size $n+k-1$. However, we do have a bona fide witness $\includegraphics[width=12mm,scale=0.5]{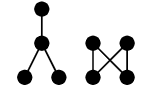}$ of cardinality $8=n+k$.  In \cref{subsec exchange graph} we will introduce a graph structure to keep track of the paths that come up in this context.  \Cref{fig exchange graph} works out this example in detail, and shows how we need a path of length four in this case.

We note in passing that this example also illustrates how some witnesses can be smaller than $n+k-1$: what was a false witness to $\Gamma_5$ is a true witness to $\Gamma_5'=\{
\includegraphics[width=10mm,scale=0.5]{C1},\includegraphics[width=8mm,scale=0.5]{C2},\includegraphics[width=8mm,scale=0.5]{C3},\includegraphics[width=5mm,scale=0.5]{C5}\}$. The intuition for why this witness has size $n+k-2<n+k-1$ is that the associated path contains only three of the four posets in $\Gamma_5$. 

In summary, the heuristic is that short paths correspond to small witnesses and long paths correspond to large witnesses.

\section{Minimal witnesses arbitrarily larger than $n+k+c$}\label{section n+k+c}

In this section we show that, for any constant $c$, there are nodes $\Gamma_n$ in covtree for which there are no witnesses smaller than $n+k+c$, where $k=|\Gamma_n|$.

We construct three examples of sequences of nodes and their minimal witnesses, where the difference between the size of the minimal witness and $n+k$ increases without bound. In each example, we begin by constructing what will ultimately be our minimal witness and then obtain the node $\Gamma_n$ that it witnesses.

The first example uses the so-called \emph{Newtonian} posets. These posets have a regular structure and are far from being generic, but provide a neat example to our claim. The second example expands the validity of the claim by allowing for arbitrary poset structures. What the two examples have in common is that the witnesses are \emph{tall} posets. The final example breaks the mould with witnesses of height 2.

For all our examples, we will need the following notation.
\begin{definition}
    For a poset $P$, we write $D(P)$ to denote the set of (non-empty) downsets of $P$ up to isomorphism. We write $D_n(P)\subseteq D(P)$ to denote those downsets of size $n$, and define $d_n(P):=|D_n(P)|$.
\end{definition}

\subsection*{Example 1: Newtonian posets}

We will need the following definition of level. For a positive integer $l$, \emph{level} $l$ in $P$ is the set of elements in $P$ of height $l$. For instance, level 1 is the set of minimal elements.\footnote{The posets we work with here are in particular graded posets so this use of the term level coincides with the usual notion of level of a graded poset up to the convention that for us minimal elements are of level $1$.}

\begin{definition}
    A \emph{Newtonian} poset \cite{Zalel:2020oyf} is one in which every element in level $i$ covers every element in level $i-1$. In other words, in a Newtonian poset, each level is complete to the level below it. 
\end{definition}

\begin{definition}
    \label{def_W} We define a poset of \emph{type \textbf{$\mathcal{W}$}} to be a disjoint union of two connected Newtonian posets that only differ in one maximal element.
\end{definition}

An example is shown in figure \ref{fig:w}.

\begin{remark}
    Note that in a Newtonian poset, all elements within a level have the same future and past (hence its name). Consequently, adding a new maximal element while retaining the Newtonian property can only be done in one of two ways: adding an element to the existing maximal level or creating a new level by putting the new element above all the existing elements. This imposes a minimum height requirement on type $\mathcal{W}$ posets: the Newtonian components must have height $\geq 2$ (for otherwise they will not be connected) and hence a type $\mathcal{W}$ poset must have height $\geq 3$.
\end{remark}
    \begin{figure}[h]
        \centering
        \begin{tikzpicture}
            \node[D](a1) at (0,0) {};
            \node[D, right of=a1](a2){};
            \node[D, right of=a2](a3){};
            \node[D,below of=a2](a4){};
            \node[D,below of=a4](a5){};
            \node[D,below left of=a5](a6){};
            \node[D,below right of=a5](a7){};
            \node[D, right of=a7](a8){};
            \node[D, left of=a6](a9){};
            \node[D, below of=a6](a10){};
            \node[D, below of=a7](a11){};
            \node[B,right of=a3](a12){};
            \draw(a12)--(a4);
            \draw(a1)--(a4)--(a5)--(a6)--(a10)--(a8)--(a11)--(a9)--(a10)--(a7)--(a11)--(a6);
            \draw(a7)--(a5);
            \draw(a2)--(a4)--(a3);
            \draw(a8)--(a5)--(a9);
            \node[D](b1) at (5,0) {};
            \node[D, right of=b1](b2){};
            \node[D, right of=b2](b3){};
            \node[D,below of=b2](b4){};
            \node[D,below of=b4](b5){};
            \node[D,below left of=b5](b6){};
            \node[D,below right of=b5](b7){};
            \node[D, right of=b7](b8){};
            \node[D, left of=b6](b9){};
            \node[D, below of=b6](b10){};
            \node[D, below of=b7](b11){};
            \node[B,above of=b2](b12){};
            \draw(b1)--(b12)--(b2)--(b12)--(b3);
            \draw(b1)--(b4)--(b5)--(b6)--(b10)--(b8)--(b11)--(b9)--(b10)--(b7)--(b11)--(b6);
            \draw(b7)--(b5);
            \draw(b2)--(b4)--(b3);
            \draw(b8)--(b5)--(b9);
        \end{tikzpicture}
        \caption{A poset of type $\mathcal{W}$ with 24 elements. The elements by which the components differ from each other are shown in blue.}\label{fig:w}
    \end{figure}
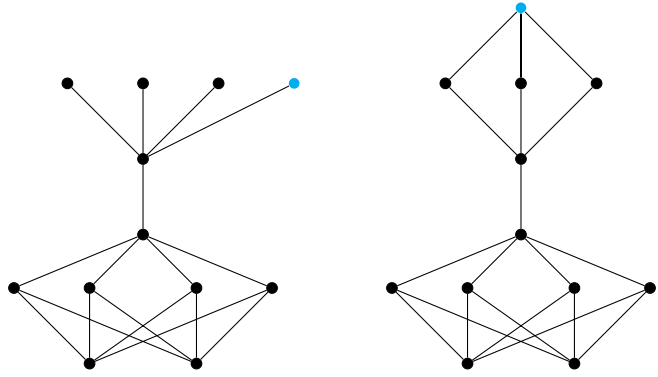

\begin{lemma}[Lemma 4.1 of \cite{Zalel:2020oyf}]
    A Newtonian poset has only one downset of each size up to isomorphism.
    \label{lemma:oneds}
\end{lemma}

\begin{proposition}
    Let $P$ denote a poset of type $\mathcal{W}$ of size $2n-2$ with at most
$n-1$ minimal elements. 
Then $d_n(P)=1+\lfloor\frac{n}{2}\rfloor$. 
\end{proposition}
\begin{proof}

    Let us denote the two disjoint Newtonian parts of the poset by $w_1$ and $w_2$. Downsets of size $n$ are obtained by taking a downset of size $n-1-l$ from $w_1$ and a downset of size $l+1$ from $w_2$, for $0\leq l\leq n-2$. Since there are at most $n-1$ minimal elements, for every value of $l$ at least one of the downsets taken from $w_1$ and $w_2$ is connected. This gives $d_n(P)=1+\lfloor\frac{n}{2}\rfloor$. 
\end{proof}

\begin{proposition}
    Let $P$ denote a poset of type $\mathcal{W}$ of size $2n-2$. Then $P$ is the unique minimal witness for $D_n(P)$.
\end{proposition}
\begin{proof}
First we argue that $P$ is a minimal witness for $D_n(P)$. Note that any minimal witness for $D_n(P)$ has to contain both $w_1$ and $w_2$, since they each appear as downsets in elements of $D_n(P)$. Consider a strict subset $Q\subset P$. Then there is a maximal element $x\in P$ that is not an element in $Q$. If $x$ is maximal in the $w_1$ part of $P$ then $w_1$ is not a downset in $Q$, and similarly for $w_2$. Then no such subset $Q$ is a witness of $D_n(P)$. Therefore, $P$ is a minimal witness of $D_n(P)$.

Now, assume for contradiction that there exists a minimal witness $Q\not= P$ for $D_n(P)$. Then the copies of $w_1$ and $w_2$ in $Q$ must be connected, for if they were disjoint then $P\subset Q$ and hence $Q$ would not be a minimal witness of $D_n(P)$. Thus, since each of $w_1$ and $w_2$ is itself connected (by definition \ref{def_W} of the type $\mathcal{W}$ poset), $Q$ contains a connected component of size $\geq n$ and therefore at least one connected downset of size $n$. But no element of $D_n(P)$ is connected. Contradiction.
\end{proof}

\begin{theorem}\label{thm section 3}
   For any constant $c$, one can find a node $\Gamma_n$ in covtree that has no witness of size $\leq n+k+c$, where $\Gamma_n=k$. 
\end{theorem}

\begin{proof}
  Fix some positive constant $i\geq 5$. Consider an infinite sequence of posets $P_i,P_{i+1},P_{i+2},\dots$ where each $P_n$ is a type $\mathcal{W}$ poset of size $2n-2$ with at most $n-1$ minimal elements. Each $P_n$ is the unique minimal witness of $\Gamma_n=D_n(P_n)$ with $k=d_n(P_n)=1+\lfloor\frac{n}{2}\rfloor$. Then for any constant $c$, we have that the unique minimal witness of $\Gamma_n$ has size larger than $n+k+c$, for all $n$ satisfying $n-\lfloor\frac{n}{2}\rfloor\geq 3+c$. Since no witness can be smaller in size than the unique minimal witness, this completes the proof.
\end{proof}

\subsection*{Example 2: Tall families with general seeds}

Now we generalise our construction to allow for generic poset substructures. Our new witnesses are tall posets built from two arbitrary seed posets $P_0$ and $R_0$. These posets are each put above a Newtonian poset of choice and the union is taken to form our new witness. We will need to count the number of downsets of this witness. For ease of counting, we will impose the restriction that our posets $P_0$ and $R_0$ do not contain each other as downsets. For clarity, we will further impose that $P_0$ and $R_0$ are of equal cardinality and that the Newtonian poset by which they are extended is a chain (also known as ladder). For completeness, a brief outline of the most general construction is given in remark \ref{remark:general construction}. 

We begin with some notation. 
\begin{definition}
    Let $P_0$ and $R_0$ denote two distinct posets of size $m$. For $i\geq 0$, let $P_i$ denote the poset that is obtained by putting $P_0$ above a chain (also known as ladder) $l_i$ of length $i$, and similarly for $R_i$ (where the chain of length $i=0$ is the empty set). Let $P_iR_i$ denote the poset obtained by taking the disjoint union of $P_i$ and $R_i$.
\end{definition}

The key to the construction is that for downsets of $P_iR_i$ that contain only the chain parts, we will generically find two copies of each downset, one taking the longer subchain from $P_i$ and one taking the longer subchain from $R_i$.  We define the following notation to keep track of these pairs.

\begin{definition}\label{def_pairs}
Define the set of all duplicates of size $s$ in $P_iR_i$ to be
    \begin{align*}
    I_{s,i}:=\bigg{\{}(p,r)\in D(P_i)\times D(R_i)\,:\,& |p|+|r|=s,|p|\geq\frac{s}{2}, p\not=r, \ 
     p,r\in D(P_i)\cap D(R_i)\bigg{\}}
    \end{align*}
\end{definition}

\begin{lemma}
    When $i>m$, the sequence of the numbers $d_{m+i+1}(P_iR_i)$ increases by 1 every other step and takes the values,    \begin{align}
        d_{m+i+1}(P_{i}R_{i})=\begin{cases}
        d_{m+i}(P_{i-1}R_{i-1})&\text{ if }m+i \text{ is even}\\
        d_{m+i}(P_{i-1}R_{i-1})+1&\text{ if }m+i \text{ is odd}.
    \end{cases}
    \end{align}
    \label{lemma:ladderextension}
\end{lemma}
\begin{proof}
    Recall that by construction, $|P_i|=|R_i|=m+i$. Then to obtain a downset of size $m+i+1$ from $P_iR_i$, we must take a downset of size $1\leq k\leq m+i$ from $P_i$ and a downset of size $m+i+1-k$ from $R_i$. Therefore, the number of downsets of size $m+i+1$ of $P_iR_i$ is given by
    \begin{equation}\label{eq formula for dPQ}
        d_{m+i+1}(P_{i}R_{i})=\sum_{k=1}^{m+i}d_{k}(P_i)d_{m+i+1-k}(R_i)-|I_{m+i+1,i}|.
    \end{equation}  
    
    To evaluate the first term, note that the downsets of size $j$ in $P_i$ are given by the downsets of size $j-1$ in $P_{i-1}$ with the added minimal element, so $d_j(P_i)=d_{j-1}(P_{i-1})$. 
    Applying this relation recursively, we get,
    \begin{align}
        d_j(P_i)=\begin{cases}
        1&\text{ if }j\leq i\\
        d_{j-i}(P_0)&\text{ if }j > i
    \end{cases}
    \end{align}
    and likewise for $R_i$.
    
    We now assume that $i>m$. Using the above formula we can split the first term of \cref{eq formula for dPQ} into three parts and simplify:
    \begin{align}\label{eR_term_1}
        &\sum_{k=1}^{m+i}d_{k}(P_i)d_{m+i+1-k}(R_i)\nonumber\\ 
        =&\sum_{k=1}^{m}1 \cdot d_{m+1-k}(R_0)+\sum_{k=m+1}^{i}1\cdot 1+\sum_{k=i+1}^{m+i}d_{k-i}(P_0)\cdot 1\\
        =&|D(P_0)|+|D(R_0)|+i-m\nonumber.
    \end{align}
    
    Now we evaluate the second term of \eqref{eq formula for dPQ}, $|I_{m+i+1,i}|$, namely the number of duplicate pairs. 
    First consider duplicate pairs that involve $P_0$ nontrivially, that is duplicate pairs of the form $(p,r)$ such that $p$ is a non-empty downset $S$ of $P_0$ above a ladder $l_i$ and $r$ is a ladder $l_k, k\leq i$. Since $(p,r)$ is a duplicate pair, it follows that $p$ is also a downset in $R_i$ and hence $S$ is downset in $R_0$. In this way, every duplicate pair of this form corresponds to some $S\in D(P_0)\cap D(R_0)$, contributing a total of $|D(P_0)\cap D(R_0)|$ to $|I_{m+i+1,i}|$. The remaining duplicate pairs (those that do not involve $P_0$ nontrivially) are of the form $(l_t,l_{m+i+1-t})$ with $t$ taking integer values in the ranges, \begin{align}
            \frac{m+i+1}{2}+1\leq t \leq i \ &\text{ if $m+i$ is odd}\nonumber\\
            \frac{m+i+1}{2}+\frac{1}{2}\leq  t \leq i \ &\text{  if $m+i$ is even,}\nonumber
    \end{align} where the lower bound follows from the requirement in definition \ref{def_pairs} that $p\geq \frac{s}{2}$ and $p\not=r$, and the upper bound follows from our assumption that $p$ is contained in the chain part of $P_i$. Enumerating the allowed values of $t$ in each case, we find that for all values of $i>m$ the number of duplicate pairs can be written as,
    \begin{align}\label{eR_term_2}
        |I_{m+i+1}|=
            |D(P_0)\cap D(R_0)|+\left\lfloor\frac{i-m}{2}\right\rfloor.
    \end{align}

    Finally, plugging \eqref{eR_term_1} and \eqref{eR_term_2} into \eqref{eq formula for dPQ} gives,
        \begin{align}\label{d:formula}
    d_{m+i+1}(P_iR_i)=C(P_0R_0)+\left\lceil \frac{i-m}{2}\right\rceil,
    \end{align}
    where we define,
\begin{align}
        C(P_0R_0):=|D(P_0)|+|D(R_0)|-|D(P_0)\cap D(R_0)|.
    \end{align}
   
Since the $i$ dependence is contained in the $\left\lceil \frac{i-m}{2}\right\rceil$ term, $d_{m+i+1}(P_iR_i)$ increases by 1 one every other step, namely when $i+m$ is odd.
\end{proof}

We illustrate lemma \ref{lemma:ladderextension} with an example. Consider,
\begin{equation}\label{eg pq}\begin{tikzpicture}[node distance= 8mm]
       \node(p)at(-2.5,-0.5){$P_0R_0=$};
        \node[D](a7) at (0,0) {};
        \node[D, below left of = a7](a5){};
        \node[D, below right of = a7](a6){};
        \node[D, below left of = a5](a1){};
        \node[D, below left of = a6](a2){};
        \node[D, below right of = a6](a3){};
        \draw(a1)--(a5)--(a7)--(a6)--(a3);
        \draw(a5)--(a2)--(a6);
        \node[D](b7) at (2.5,0) {};
        \node[D, below of = b7,node distance=6mm](b5){};
        \node[D, right of = b5,node distance=1cm](b6){};
        \node[D, below left of = b5](b1){};
        \node[D, below left of = b6](b2){};
        \node[D, below right of = b6](b3){};
        \draw(b1)--(b5)--(b7);
        \draw(b6)--(b3);
        \draw(b5)--(b2)--(b6);
    \end{tikzpicture}
\end{equation}
    and count their downsets to obtain $
    C(P_0R_0)=7+9-6=10$. 
 
Now extend the seed posets by a ladder of length $i=7$. The number of downsets of $P_7R_7$ of cardinality $m+i+1=14$ is given by \eqref{d:formula} as $d_{14}(P_7R_7)=11$.
We can confirm this by constructing the downsets explicitly to obtain,

\hspace{-10mm} \includegraphics[height=.95in]{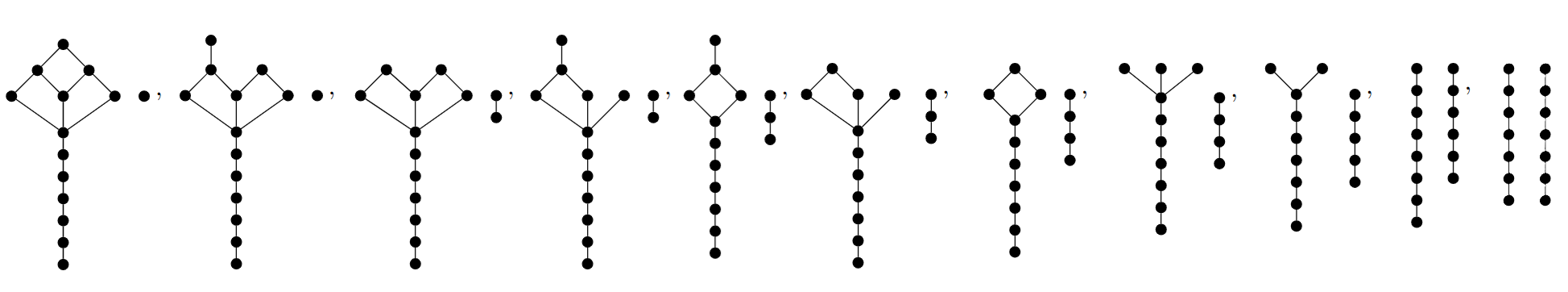}.

\begin{lemma}\label{lem PiQi min}
    Suppose $P_0$ and $R_0$ are distinct and $i\geq1$. Then no witness to $D_{i+m+1}(P_iR_i)$ has size less than $P_iR_i$.
\end{lemma}

\begin{proof}
    First, note that no poset in $D_{i+m+1}(P_iR_i)$ is connected because the largest connected component of $P_iR_i$ has size $m+i$.
    
    Suppose $Q$ is a witness to $D_{i+m+1}(P_iR_i)$ of size less than $P_iR_i$. Since we have both $P_il_1$ and $l_1R_i$ in $D_{i+m+1}(P_iR_i)$, $Q$ must have both $P_i$ and $R_i$ as downsets.  Since $Q$ has size less than $P_iR_i$ there must be a copy of $P_i$ and a copy of $R_i$ as downsets of $Q$ whose union in $Q$ is connected; let $S$ be this union.  Since the intersection of these copies of $P_i$ and $R_i$ must be a downset of each of them it must in particular have a unique minimal element.  Consequently, $S$ is connected and has a unique minimal element and so has at least one connected downset of every size from $1$ to $|S|$.  However, $P_0$ and $R_0$ are distinct and so $|S|\geq m+i+1$.  Thus $S$ and hence also $Q$ has a connected downset of size $m+i+1$, contradicting the fact that $D_{m+i+1}(P_iR_i)$ does not contain any connected posets.
\end{proof}

\begin{proposition}
    Fix some distinct $P_0$ and $R_0$. Then for any $i>m$, $D_{m+i+1}(P_iR_i)$ is a node in covtree with $k=C(P_0R_0)+\left\lceil \frac{i-m}{2}\right\rceil$ elements of size $n=m+i+1$ and with a smallest witness of size $|P_iR_i|=2m+2i$. Then for any constant $c$, we have that $|P_iR_i|$ is strictly greater than $n+k+c$ for $i-\lceil\frac{i-m}{2}\rceil>1+C(P_0R_0)-m+c$.
\end{proposition}

\begin{proof}
    \cref{lemma:ladderextension} and \cref{lem PiQi min} give the first part. The second part is immediate from the construction.\end{proof}

\begin{remark}\label{remark:general construction}
    We briefly outline the general construction: consider two posets $P_0$ and $R_0$ with $p:=|P_0|\leq|R_0|$ and $D_p(P_0)\cap D_p(R_0)=\varnothing$. Consider the poset obtained by extending $P_0$ and $R_0$ by a connected Newtonian poset of size $l$ downwards and taking their disjoint union. Denote this by $P_lR_l$. If we now consider its downsets of size $n$ we see that, provided $n$ is large enough and the number of minimal elements of $P_lR_l$ is small enough so that every downset of size $n$ must involve both components of $P_lR_l$ and the parts from both sides will be connected, then we can count by the same methods as the examples above.  Thus, with a suitable constraint on the number of minimal elements in the Newtonian part and a suitable relationship between $n$, $l$ and  the size of the Newtonian part, we again obtain a family of examples $D_n(P_lR_l)$ with smallest witness $P_lR_l$ so that for any $c$, $|R_lR_l|$ is eventually larger than $n+k+c$.

\end{remark}
\begin{remark}\label{rem order dim}
    One might hope that putting a restriction on the posets might enable us to regain $n+k-1$ as a bound for the size of a minimal witness, as it did in the case $k=2$.  Equation \eqref{eg pq} shows that restricting to order dimension 2 will not suffice to bring the bound to $n+k-1$, since the posets in that example have order dimension 2.  
    A poset has \emph{order dimension} $p$ if it can be written as an intersection of $p$ total orders, where the intersection of posets on a common underlying set is the poset on that underlying set where one element is greater than another if and only if they are so related in all of the posets being intersected.  Order dimension 2 is particularly special in the context of causal set theory since it can be faithfully embedded in 2d Minkowski spacetime, see for instance \cite{Glaser:2017sbe}.
\end{remark}

\subsection*{Example 3: A family of height 2}
In this section we show that restricting to height 2 posets is not sufficient to regain any bound of the form $n+k+c$ for a constant $c$.

\begin{definition}
    Let $\Lambda_i$ be the poset with $i$ minimal elements and one maximal element above all the minimal elements.
    
    Let $H_\ell$ be the poset whose connected components are $\ell$ copies of $\Lambda_2$ and one copy of $\Lambda_{3(\ell-1)}$ as illustrated in \cref{fig height2family}.
\end{definition}

\begin{figure}[h]
    \includegraphics{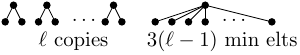}
    \caption{A poset with few downsets of size $3\ell$ relative to its size.}\label{fig height2family}
\end{figure}

\begin{lemma}\label{lem k in the height 2 eg}
    $d_{3\ell}(H_\ell) = \ell+2$.
\end{lemma}

\begin{proof}
    Note that if a downset contains the maximal element of some $\Lambda_{i}$ then it contains the entire $\Lambda_i$.  In particular, if a downset of size $3\ell$ of $H_\ell$ contains the maximal element of the $\Lambda_{3(\ell-1)}$ then since $|\Lambda_{3(\ell-1))}| = 3\ell-2$, this downset must be $\Lambda_{3(\ell-1)}$ along with 2 incomparable elements.  
    The remaining downsets of size $3\ell$ of $H_\ell$ are thus determined by how many $\Lambda_2$s they contain with the rest being incomparable elements.  There can be between $0$ and $\ell$ (inclusive) copies of $\Lambda_2$ and all such values are possible.  In total this gives $\ell+2$ downsets of size $3\ell$.
\end{proof}

\begin{lemma}\label{lem witness size in height 2 eg}
    No witness to $D_{3\ell}(H_\ell)$ has size less than $|H_\ell| = 6\ell-2$. 
\end{lemma}

\begin{proof}
    In view of the downsets identified above, any witness must contain $\ell$ disjoint copies of $\Lambda_2$ and must contain a copy of $\Lambda_{3(\ell-1)}$.  To be smaller than $H_\ell$, a copy of $\Lambda_2$ must overlap with either another $\Lambda_2$ or a $\Lambda_{3(\ell-1)}$, while still maintaining each as a downset.  Consequently, they must share minimal elements, however the four ways of doing so: overlapping two $\Lambda_2$ at two minimal elements or at one minimal element, or overlapping a $\Lambda_2$ and a $\Lambda_{3(\ell-1)}$ at two minimal elements or at one mimial element, each become a downset in the purported witness and are of of size $4, 5, 3\ell - 1$ or $3\ell$ respectively. Enough minimal elements remain from the other $\Lambda_i$s to expand each of these into a downset of size $3\ell$, but none of these downsets are in $D_{3\ell}(H_\ell)$ giving a contradiction. 
\end{proof}

\begin{proposition}\label{thm height 2}
    $D_{3\ell}(H_\ell)$ is a node of covtree with $k=\ell+2$ elements of size $n=3\ell$ and with smallest witness size $|H_\ell| = 6\ell-2$.  Furthermore, all the posets of $D_{3\ell}(H_\ell)$ as well as $H_\ell$ itself have height $2$ and for any constant $c$, we have that $|H_\ell|$ is  strictly greater than $n+k+c = 4\ell+2+c$ for $\ell>2+c/2$.
\end{proposition}

\begin{proof}
    \cref{lem k in the height 2 eg} and \cref{lem witness size in height 2 eg} give the first part.  The furthermore is immediate by construction.
\end{proof}

\section{The exchange graph}\label{subsec exchange graph}

Here we introduce the notion of the \emph{exchange graph} and apply it to obtain the following improved bound on witness size. Recall that $Q$ is a minimal witness to $\Gamma_n$ if no proper downset of $Q$ is itself a witness to $\Gamma_n$.

\begin{theorem}
    \label{lemma: improved bound general}
    If $Q$ is a minimal witness to $\Gamma_n=\{P_1,\ldots ,P_k\}$, with $k\geq 3$, then $|Q|\leq n(k-1)$.
\end{theorem}

To prove theorem \ref{lemma: improved bound general}, we begin with a definition.

\begin{definition} The \emph{exchange graph} of size $n$ downsets of a poset $Q$ is a graph $G_n(Q) = (V, E)$ where $V$ is the set of downsets of size $n$ of $Q$ and for $A, B \in V$ there is an edge from $A$ to $B$ if $A$ and $B$ are downsets that differ by exactly one element. \label{def:graph}
\end{definition}

Note that in $G_n(Q)$, the vertices are not the downsets of size $n$ \emph{up to isomprphism}, but rather \emph{all} downsets of size $n$.  The edge relation requires remembering how the downsets lived in $Q$.  We may have many vertices of $G_n(Q)$ which are isomorphic to each other.

    Examples are shown in \cref{triangles in exchange graph} and \cref{fig exchange graph}.
 \begin{figure}[h]
        \includegraphics{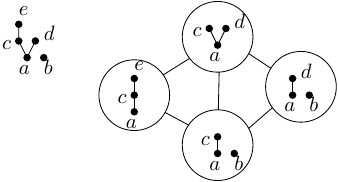}
        \caption{Exchange graph for the poset on the left with the labelling as given.  This example will be useful for illustrating how short paths sometimes need special consideration, cf. \cref{rem m<3 is special}.}\label{triangles in exchange graph}
    \end{figure}

    \begin{figure}[h]
        \includegraphics{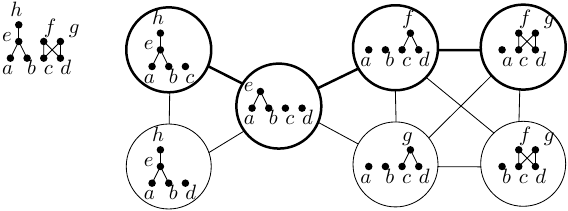}
        \caption{Exchange graph $G_5$ for the poset on the left with the labelling as given.  The bolded path in the exchange graph gives a minimal $A-B-\cdots-B-C$ path as discussed later in this section, cf. setup to \cref{lem A cup B general}.}\label{fig exchange graph}
    \end{figure}

    \begin{remark}
In the case that the poset $Q$ has order dimension 2 (see \cref{rem order dim}), $G_n(Q)$ can be upgraded from a graph to an acyclic diagraph using the order on the space dimension. This is proved and discussed in more detail in \cite{Smmath}.

\end{remark}

The following proposition will be of use.

\begin{proposition}
    \label{thm diameter} For any poset $Q$ and positive integer $n\leq |Q|$, the graph $G_n(Q)$ is connected.  Furthermore, for $A, B \in V$, the distance between $A$ and $B$ is $|A\setminus (A\cap B)|$, and so in particular the diameter of $G_n(Q)$ is at most $n$.
\end{proposition}

\begin{proof}

We proceed by induction on $|A\setminus (A\cap B)|$. Note that if $|A \setminus (A \cap B)| = 1$ then $A$ and $B$ are adjacent in $G_n(Q)$. Suppose that there exists a path of length $k-1$ between any pair of vertices, $A$ and $B$, satisfying $|A \setminus (A \cap B)| = k-1$ for some $k\geq 2$. We will now show that when $|A \setminus (A \cap B)| = k$ there exists a vertex $A'$ such that $|A'\setminus(A'\cap A)|=1$ and $|A'\setminus(A'\cap B)|=k-1$, and the result follows.

Let $X = A \cap B$, $\{a_1, \dots, a_k\}=A\setminus X$ and $\{b_1, \dots, b_k\}=B\setminus X$. Without loss of generality, let $a_k$ be maximal in $A\setminus X$ (hence also maximal in $A$) and $b_k$ be minimal in $B\setminus X$. Then $A' = X\cup \{b_k\} \cup \{a_1, \dots, a_{k-1}\}$ is a vertex in $G_n(Q)$ satisfying $|A' \setminus (A' \cap B)| = k-1$ and $|A' \setminus (A' \cap A)| = 1$.

Finally, the previous paragraphs show that the distance between $A$ and $B$ is at most $|A\setminus (A\cap B)|$ (and as this holds for any $A$, $B$ the graph is in particular connected), but this distance is also at least $|A\setminus (A\cap B)|$ since one must exchange at least that many elements to get from $A$ to $B$.

\end{proof}

The connectivity of $G_n(Q)$ already gives us a slightly improved upper bound on the size of a minimal witness $Q$ to $\Gamma_n=\{P_1, \ldots, P_k\}$ than the known bound of $nk$, as stated in the following lemma. 

\begin{lemma}\label{lemma: improved_bound}
    If $Q$ is a minimal witness to $\Gamma_n=\{P_1, \ldots, P_k\}$  with $k\geq 2$, then $|Q|\leq n(k-1)+1$.
\end{lemma}

\begin{proof}
   Let $Q$ be a minimal witness of $\Gamma_n$. Since $G_n(Q)$ is connected, there exists a copy of $P_i$ that is adjacent to some copy of $P_j$, for some $1\leq i\neq j\leq k$.  The union of these copies of $P_i$ and $P_j$ is a downset of size $n+1$ in $Q$. Taking the union of this downset with a copy of each of the other $P_\ell$ in $Q$ enlarges it to a new downset $Q'$ of size at most $n+1+n(k-2) = n(k-1)+1$. Now, $Q'$ is itself a witness to $\Gamma_n$ as it contains each of the posets in $\Gamma_n$ as a downset. Therefore, if $Q$ is a minimal witness then $Q=Q'$ and the result follows.
\end{proof}

The bound in lemma \ref{lemma: improved_bound} is saturated when $k=2$ where it is known that every minimal witnesses is of size $n+1$ (cf. \cref{sec background} and \cite{Zalel:2020oyf}).

Now we use the connectivity of the exchange graph to further improve on this bound in the case $k\geq 3$ and obtain theorem \ref{lemma: improved bound general}. This may seem like a lot of effort for an improvement of $1$, but these tools will also be foundational for our more detailed consideration of the case $k=3$ in \cref{subsec k=3}.

Let $Q$ be a witness to $\Gamma_n=\{P_1, \ldots, P_k\}$ with $k\geq 3$. Since $G_n(Q)$ is connected (\cref{thm diameter}) we know that there are two distinct $P_j$ in $\Gamma_n$ for which there are adjacent copies of these posets in $G_n(Q)$; without loss of generality say there are copies of $P_1$ and $P_2$ adjacent in $G_n(Q)$. Consider all paths in $G_n(Q)$ from an adjacent $P_1-P_2$ pair to a downset of size $n$ which is not a copy of $P_1$ or $P_2$.  Take a shortest such path $\mathcal{P}$.  Note that there cannot be any internal vertex along $\mathcal{P}$ that is not isomorphic to $P_1$ or $P_2$ for otherwise $\mathcal{P}$ would not be a shortest path. Furthermore, if after the initial $P_1-P_2$ pair there is ever a switch between $P_1$ and $P_2$ along $\mathcal{P}$ then this other adjacent $P_1-P_2$ pair along the path is closer to the downset which is not a copy of $P_1$ or $P_2$, so $\mathcal{P}$ was again not a shortest path.

Therefore, reindexing as necessary, we can without loss of generality assume we have a path $\mathcal{P}$ of the form $P_1-P_2-P_2-\cdots-P_2-P_3$ in $G_n(Q)$ and no shorter path of this form for any three distinct $P_i$ exists in $G_n(Q)$. The copies of $P_1$ and $P_3$ at the ends of the path we will call the \emph{anchor copies} of $P_1$ and $P_3$.

\begin{lemma}\label{lem A cup B general}
    Take assumptions and notation as above.  Let $R$ be the union of the shortest path $\mathcal{P}=P_1-P_2-\cdots-P_2-P_3$ defined above.  Further assume there are at least two copies of $P_2$ in the shortest path. Then
    \begin{enumerate}
        \item  $R=P_1\cup P_3$, and 
        \item the number of copies of $P_2$ in the path is equal to $|P_1 \setminus (P_1\cap P_3)|-1=|P_3 \setminus (P_1\cap P_3)|-1$.
    \end{enumerate} 
\end{lemma}

\begin{proof}
    Let $X$ be the intersection of the anchor copies of $P_1$ and $P_3$.  Unless otherwise specified when we refer to $P_1$ and $P_3$ below we will specifically mean the anchor copies. 
    
    By \cref{thm diameter}, there exists a path $\mathcal{P}'$ in $G_n(Q)$ from the anchor copy of $P_1$ to the anchor copy of $P_3$ that is built by successively removing one element of $P_1\setminus X$ and adding one element of $P_3\setminus X$. The $\mathcal{P}$ can be no shorter than $\mathcal{P}'$ since every element of $P_1\setminus X$ must be removed and every element of $P_3\setminus X$ must be added. 
    
    Suppose $\mathcal{P}'$ were strictly shorter than $\mathcal{P}$.  By the minimality of $\mathcal{P}$ this implies that $\mathcal{P}'$ is not of the form discussed above, nor is any subpath of it of that form for any three distinct $P_i$s.  The only way this can happen is if $\mathcal{P}'$ contains only copies of $P_1$ and $P_3$.  We now need to consider two cases.  If $\mathcal{P}'$ is simply $P_1-P_3$, that is the anchor copies are adjacent, then taking $\mathcal{P}'$ along with the first copy of $P_2$ in $\mathcal{P}$ gives a path of the form $P_2-P_1-P_3$ which by the assumption that there are at least two copies of $P_2$ in $\mathcal{P}$ contradicts the minimality of $\mathcal{P}$.  Now suppose $\mathcal{P}'$ has at least one internal node.  In this case either the portion of $\mathcal{P}'$ from the anchor copy of $P_1$ to the first appearance of a copy of $P_3$ in $\mathcal{P}'$ or the portion of $\mathcal{P}'$ from the anchor copy of $P_3$ to the first appearance of a copy of $P_1$ in $\mathcal{P}'$ is strictly shorter than $\mathcal{P}'$ itself (potentially both are strictly shorter).  Call these portions $\mathcal{P}'_1$ and $\mathcal{P}'_3$, then at least one of $\mathcal{P}'_1$ along with the first $P_2$ in $\mathcal{P}$ or $\mathcal{P}'_3$ along with the last $P_2$ in $\mathcal{P}$ gives a strictly shorter path than $\mathcal{P}$ of the special form, thus contradicting the minimality of $\mathcal{P}$.  
    
    In all cases we have contradicted that $\mathcal{P}'$ is strictly shorter than $\mathcal{P}$ so they must have the same length. Therefore $\mathcal{P}$ must itself be formed by successively removing one element of $P_1\setminus X$ and adding one element of $P_3\setminus X$. Hence, for every $P_2$ in the path, each of its elements is contained in one of the anchor copies, hence $R=P_1\cup P_2$. Furthermore, each copy of $P_2$ in the path corresponds to swapping one element of $P_1$ with an element of $P_3$, hence the number of copies of $P_2$ in $\mathcal{P}$ is $|P_1\setminus X|-1$.
    
\end{proof}

Finally, theorem \ref{lemma: improved bound general} is a corollary.

\begin{proof}[Proof of theorem \ref{lemma: improved bound general}]
    Given a witness $Q$, take a path $P_1-P_2-\cdots -P_2-P_3$ as discussed immediately before \cref{lem A cup B general}.  
     Let $Q'$ be a downset of $Q$ given by the union of this path along with one copy of each $P_j$ for $j>3$.  Then $Q'$ is also a witness for $\Gamma_n$.
    
    If there are not at least two copies of $P_2$ in the path then $|Q'| \leq n+2 + (k-3)n = n(k-2)+ 2 < n(k-1)$  since $k\geq 3$ implies $n>2$.    If there are at least two copies of $P_2$ in the path, by \cref{lem A cup B general}, $|Q'| \leq 2n + (k-3)n = n(k-1)$.
    
    If $Q$ is minimal then $Q=Q'$ and the result follows.
\end{proof}

\begin{remark}\label{rem m<3 is special}
    In the previous two proofs we saw that the situation where there was only one $P_2$ in the path was special.  We can see why this is in the example of \cref{triangles in exchange graph}.  The path from the poset on $a,c,e$ to the poset on $a,c,d$ to the poset on $a,c,b$ is a minimal path of the form discussed immediately before \cref{lem A cup B general}, however, the union of the anchor copies is not the union of the path because it does not include $d$.  The problem here is that the anchor copies are adjacent.  The only way this can happen without contradicting minimality is if there is a triangle in the exchange graph where the three vertices of the triangle are non-isomorphic downsets --- for any longer $P_1-P_2-\cdots-P_2-P_3$ path, if $P_1$ and $P_3$ were adjacent then taking the anchor copies along with the first $P_2$ in the path would give $P_3-P_1-P_2$ of the desired form but strictly shorter.

    Since our overall goal is to show there must be small witnesses, configurations with many non-isomorphic downsets close together in the exchange graph are helpful rather than problematic, but the slightly different behaviour when the minimal path has only one $P_2$ does sometimes require special consideration.
\end{remark}

\section{The case $k=3$}\label{subsec k=3}

We now turn to the special case $k=3$. Here, instead of bounding the size of \emph{all} minimal witnesses from above, we give an upper bound that is respected by \emph{at least one} minimal witness for each node. To keep the notation lighter in subscripts, we will denote our nodes by $\Gamma_n=\{A, B, C\}$. Our main result is the following.

\begin{theorem}\label{thm k=3}
     If $\Gamma_n=\{A,B,C\}$ has a witness then it has at least one minimal witness $Q$ with $|Q|\leq\frac{3}{2}(n+1)$.
\end{theorem}

By loosening the generality of the statement from \emph{all} witnesses to \emph{at least one} witness we are able to tighten the bound and gain much in the way of constructing covtree because to ascertain that some $\Gamma_n$ is a node in covtree one need only find one, not all, of its witnesses. In particular, when constructing level $n$ of covtree by brute force (for lack of a current alternative), one sifts through all finite posets up to a certain size, for each constructing its set of downsets of size $n$. Each such set is a node at level $n$, and a set that cannot be obtained in this way is (by definition) not a node. Theorem \ref{lemma: improved bound general} tells us that to construct the nodes at level $n$ with $k=3$ it suffices to consider posets no larger than $2n$, while theorem \ref{thm k=3} says that one can stop earlier at size $\frac{3}{2}(n+1)$. To get an idea of numbers, at $n=10$ the bounds of theorems \ref{lemma: improved bound general} and \ref{thm k=3} are 20 and 16, respectively, meaning 9349 fewer posets to search through---and by $n=15$ this number grows to 1,271,244.

In the remainder of this section, we prove theorem \ref{thm k=3} through a series of lemmas.  The basic structure is as follows.  First we give the general setup which is based on the minimal path as in \cref{lem A cup B general}. We will denote the anchor copies in the path by $A$ and $C$. \Cref{lem all B} gives us more information about this path. In \cref{lem stick condition} and \cref{lem only bits}, we work out some significant constraints on the poset structure of the set difference between the anchor copies. The final step then is to bound the size of $|A\setminus (A\cap C)|$ in term of the size of $|A\cap C|$. \Cref{lem max elt structure} establishes some structure that will be useful for this and then \cref{lem all covers all} along with the proof of the theorem at the end of the subsection completes the plan, establishing these bounds by breaking the problem into two cases.  The case dealt with in \cref{lem all covers all} is the one case where we may need to go outside our initial $Q$, which is why the theorem only tells us that there exists a minimal witness with size under the bound, but does not tell us about all minimal witnesses.  

\

Let $Q$ be a witness of $\Gamma_n=\{A,B,C\}$. Then as in the set up to \cref{lem A cup B general}, there exists a shortest path $\mathcal{P}$ in $G_n(Q)$ of the form $A-B-B-\cdots-B-C$. The copies of $A$ and $C$ at the ends of the path are, again, our anchor copies, and we will use the letters $A$ and $C$ to refer to these copies specifically in what follows unless otherwise specified. We will assume that $Q$ is the union of this path. In particular, this means that $Q$ is the union of the anchor copies: $Q=A\cup C$. 

Let $X=A\cap C$ and let $m=|A\setminus X| = |C\setminus X|$.  The goal is to show that $m$ must be small, following the outline described above, because $|Q|=n+m$ so bounding $m$ from above bounds $|Q|$ from above.   Note that \cref{fig exchange graph} shows that some $\Gamma_n=\{A,B,C\}$ require $m=3$, and so there cannot be a universal upper bound for $m$ that is less than $3$.  Consequently, any time we have some particular situation where $m< 3$ then such a situation is already below any universal bound we could hope for, and so can be considered to be done.  Thus, when convenient, we will assume $m\geq 3$ to avoid the situation discussed in \cref{rem m<3 is special}, and a bit later in the section we will assume $m>3$ so as to simplify the application of a later lemma (\cref{lem only bits}).
We suspect that $3$ is in fact an upper bound for $m$ in the case $k=3$---that is, for any node $\Gamma_n$ with $k=3$ there exists at least one witness $Q$ satisfying $|Q| \leq n + 3$---but we have not been able to prove this.

\begin{lemma}\label{lem all B}
   With assumptions and notation as above, and assuming $m\geq 3$, every downset of size $n$ of $Q$ which does not contain all of $A\setminus X$ nor all of $C\setminus X$ is isomorphic to $B$.
\end{lemma}

Note that this is not a trivial statement since some downsets of size $n$ of $Q$ could be made of a mix of elements of $A$ and $C$ without being any of the nodes in the path.  The content of the statement is that all these downsets of size $n$ are also isomorphic to $B$.

\begin{proof}
    Let $S\subset Q$ be a downset of size $n$ that does not contain all of $A\setminus X$ nor all of $C\setminus X$. 
    Let $m_1=|S\cap(A\setminus X)|$, $m_2=|S\cap(C\setminus X)|$, and so by hypothesis, $m_1<m$ and $m_2<m$.
    
    Every downset of $Q$ is isomorphic to $A$, $B$, or $C$.  Suppose for a contradiction that $S$ is not isomorphic to $B$.  By removing the elements of $S\cap (A\setminus X)$ from $S$ and adding in the missing elements of the anchor copy of $C$ one by one we obtain a path from $S$ to $C$ of length $m_1<m$, and similarly we obtain a path from $S$ to $A$ of length $m_2<m$. If $S$ is isomorphic to $A$ then this $S$ to $C$ path gives a shorter path from $A$ to $C$ than $\mathcal{P}$ and similarly for the $S$ to $A$ path if $S$ is isomorphic to $C$.  In either case we get a shorter $A$ to $C$ path than $\mathcal{P}$.  Since $m\geq 3$ there are at least two copies of $B$ in $\mathcal{P}$ so as in the proof of \cref{lem A cup B general}, there is a shorter path of the special form, contradicting the minimality of $\mathcal{P}$. 
\end{proof}

Note that the multiset of heights of all elements of a poset is an isomorphism invariant of the poset.

\begin{lemma}\label{lem stick condition}
    It is not possible for $A\setminus X$ to be partitioned as $A\setminus X = \{a', a\}\cup A_1\cup A_2$ where $a'$ covers $a$, all elements of $A_1$ are above or incomparable to $a'$, and all elements of $A_2$ are below or incomparable to $a$, with $A_1$ and $A_2$ both nonempty.  It is also not possible for $C\setminus X$ to be likewise partitioned.
\end{lemma}

Note that the partition in lemma \ref{lem stick condition} is generic and not unique (since elements incomparable to both $a$ and $a'$ can be contained in either $A_1$ or $A_2$). Indeed, most posets will allow for at least one such partition. Therefore, the statement that our posets of interest do not allow for any such partition will pose a powerful constraint on their structure.

\begin{proof}
    If $m\leq 3$ then no partition as described in the statement is possible, so the lemma holds.  Assume then that $m>3$.
    
    Suppose $A\setminus X$ can be partitioned as described in the statement. Take $X\cup A_2\cup \{a, a'\}$ and add elements of $C\setminus X$ so as to obtain a downset $B'$ of size $n$ in $Q$.  Since $A_1$ is nonempty, by \cref{lem all B}, $B'$ is a copy of $B$.  Remove $a'$ and add an additional element of $C\setminus X$, call it $c$, to obtain $B''$ which is likewise a copy of $B$.  Finally remove $a$ and add a further element of $C\setminus X$, call it $c'$, to obtain $B'''$.  Since $A_2$ is nonempty, we also have that \cref{lem all B} implies that $B'''$ is a copy of $B$. 

    Since the multiset of heights is an isomorphism invariant, comparing $B'$ and $B''$, we have that $h(a') = h(c)$ and comparing $B''$ and $B'''$, we have that $h(a) = h(c')$.  We also have $h(a') > h(a)$ since $a'$ covers $a$, so $h(c)> h(c')$.  In particular this means that $c'$ is not above $c$ in the poset.  Hence $B'''' = B''\cup \{c'\} \setminus \{c\}$ is also a downset of $Q$ and by the same reasoning as before is also a copy of $B$.  Comparing the heights of elements in $B'$ and $B''''$ gives us that $h(a') = h(c')$ and between $B'''$ and $B''''$ gives $h(a) = h(c)$ hence $h(c')>h(c)$ which is a contradition.

    The result for $C\setminus X$ follows by switching the roles of $A$ and $C$.
\end{proof}

\begin{lemma}\label{lem only bits}
    Only the following posets are possible for $A\setminus X$ and $C\setminus X$: an antichain of any size along with potentially a single element beneath all elements of the antichain and potentially a single element above all elements of the antichain (potentially both) or the three element poset with a chain of two elements and one incomparable element.
\end{lemma}

\begin{proof}
    We seek to avoid the configurations of \cref{lem stick condition}. Note that all posets with at most three elements are permissible.
    
    Consider a poset with 4 or more elements. If there are no cover relations at all, then the poset is an antichain.  Suppose there is some $a'$ covering $a$.  By \cref{lem stick condition} we know that there can be no partition of $A\setminus X\setminus \{a,a'\}$ into $A_1$ and $A_2$, both non-empty, with all elements of $A_1$ above or incomparable to $a'$, and all elements of $A_2$ below or incomparable to $a$. That is, if any element is above $a'$ then no element can be below or incomparable to $a$, and if any element is below $a$ then no element can be above or incomparable to $a'$, and if any element is incomparable to $a$ or $a'$, then no second element can be incomparable to the other one of them.  Since there are at least two elements other than $a$ and $a'$ in $A\setminus X$, this implies that either $a'$ is a unique maximal element or $a$ is a unique minimal element. This is true for all cover relations in the poset, so the poset has height at most 3.  Furthermore, if it has height 3 then it has both a unique maximal and a unique minimal element, while if it has height 2 then it has either a unique maximal or a unique minimal element.  This gives the posets described in the first part of the statement.
\end{proof}

As discussed above, we know that a universal upper bound for $m$ cannot be less than $3$, and so any specific situation with $m\leq 3$ is already covered by any universal bound and hence does not need further consideration.  Consequently, for the remainder of this section we will assume $m>3$ and hence only need to concern ourselves with the first possibility in the statement of \cref{lem only bits}.

\begin{lemma}\label{lem max elt structure}
    With assumptions as above, consider the poset obtained by removing one maximal element of $A\setminus X$ and one maximal element of $C\setminus X$ from $Q$.  Then,
    \begin{enumerate}
        \item all maximal elements of this poset are of the same height, and
        \item every element of this poset that is covered by at least one of the maximal elements is covered by at least $m-2$ of the maximal elements.
    \end{enumerate}
\end{lemma}

\begin{proof}
    Let the poset in the statement be $Y$.  Note that $Y$ has size $n+m-2$ and that $Y$ has at least $2m-4$ maximal elements from the remaining elements of $A\setminus X$ and $C\setminus X$ (since $A\setminus X$ and $C\setminus X$ must be configured as in the first part of \cref{lem only bits}).
    
    Every downset of size $n$ of $Y$ is a copy of $B$ by \cref{lem all B}, so in particular removing any $m-2$ maximal elements of $Y$ gives a copy of $B$.

    Let $h$ be the largest height of any element of $Y$; all elements of that height are necessarily maximal. To prove part (1) of the claim, suppose for contradiction that $Y$ has $i_1$ maximal elements of height $h$ and $i_2>0$ maximal elements of heights less than $h$.
    
    Then, by taking $Y$ and removing $\min\{i_1, m-2\}$ maximal elements of height $h$ and $m-2-\min\{i_1, m-2\}$ maximal elements of other heights, we obtain a copy of $B$ with $i_1 - \min\{i_1, m-2\}$ elements of height $h$. Taking $Y$ again and removing $\min\{i_1, m-2\}-1$ maximal elements of height $h$ and $m-1-\min\{i_1, m-2\}$ maximal elements of other heights we obtain a copy of $B$ with $i_1 +1 - \min\{i_1, m-2\}$ elements of height $h$, which is a contradiction. This proves part (1).

    Now consider an element $b$ in $Y$ that is covered by at least one of the maximal elements and suppose for a contradiction that it is covered by $m-3$ or fewer elements.  Then removing $m-3$ maximal elements, including all of those covering $b$, and removing $b$ itself gives a copy of $B$, as does removing $m-2$ maximal elements.  However, $b$ has height strictly less than $h$, so the multiset of heights differs between these copies of $B$ giving a contradiction and hence proving part (2).
\end{proof}

The next lemma is the final step before the proof to our theorem. In the theorem proof we will bound $m$ from below, and this lemma provides the result in one of the special cases that we will need to consider. This lemma is the first and only place in this work where we need to consider a witness that is smallest in size, not just minimal in the poset of witnesses. To that end, we additionally assume for the rest of the subsection that the $A-B-\cdots-B-C$ path we are working with is not just the shortest path in $G_n(Q)$ for some initial $Q$, but is the shortest path among such paths in the exchange graph of \emph{any} witness. Since the length of the path is $m-1$ and the size of the associated witness is $|Q|=n+m$, this choice of path guarantees that our $Q$ is the smallest witness to our node. 

\begin{lemma}\label{lem all covers all}
With assumptions as above,
    let $h$ be the height of the maximal elements of the subposet of $Q$ described in \cref{lem max elt structure}. Suppose that every element in $Q$ that is covered by at least one element of height $h$ is covered by all of the elements of height $h$.  Then $m\leq \frac{1}{2}(n+1)$.
\end{lemma}

\begin{proof}
   Recall that $|X|=n-m$. Hence $m\leq \frac{1}{2}(n+1)\iff  |X|\geq m-1$. We will show that either $|X|\geq m-1$ or there exists a smaller witness than $Q$, which is a contradiction.  

    Let $Y$ be the poset described in \cref{lem max elt structure}, that is the poset resulting from removing a maximal element $a$ of $A\setminus X$ and a maximal element $c$ of $C\setminus X$ from $Q$, \textit{i.e.}, $Y=Q\setminus \{a,c\}$. Under the given hypothesis, the maximal elements of $Y$ are indistinguishable since they all cover the same elements. 
    
    We now construct a new poset $Q'$ from $Y$ in the following way. Suppose $j$ of the maximal elements were from $C\setminus X$, call them $c_1, c_2, \ldots, c_j$.  Remove $c_j$.  Take one of the maximal elements of $Y$ that was from $A\setminus X$, call it $a_j$.   Since all of the maximal elements of $Y$ are indistinguishable, the set $\{c_1, \ldots, c_{j}\}$ is indistinguishable from $\{c_1, \ldots, c_{j-1}, a_j\}$.  Now add a new element $c'$ to form a new downset isomorphic to $C$, call this downset $C'$, where this new element $c'$ is above $\{c_1, \ldots, c_{j-1}, a_j\}$ if $C\setminus X$ had a maximal element and is another indistinguishable maximal element like those in $Y$ otherwise, guaranteeing that $((Y\cap C)\setminus\{c_j\})\cup \{c', a_j\} = C'$ is a copy of $C$.  Let $Q' = A\cup C'$. 
   
    Note that $|Q'| = |Y|+2-1 = |Q|-1$ and that $Q'$ has at least one copy of $A$, $B$ and $C$ as downsets and satisfies the covering condition of the statement.  If $Q'$ has no downset of size $n$ not isomorphic to $A$, $B$, or $C$, then we have a smaller witness satisfying the hypotheses of the statement, contradicting that $Q$ is witness of smallest size.  
    
    Now assume that $Q'$ has a downset $D$ of size $n$ which is not isomorphic to $A$, $B$, or $C$. 
    
    Any downset of $Q'$ of size $n$ not containing $c'$ is a downset of $Q$ as well and is therefore isomorphic to $A$, $B$, or $C$. By the indistinguishability of the maximal elements of $Y$, there exists a bijection $f:Y\setminus \{a_j\}\rightarrow Y\setminus \{c_j\}$ that takes $c_j$ to $a_j$ and fixes everything else.  Consequently, there exists a bijection $g:Q\setminus \{a, a_j\}\rightarrow Q'\setminus \{a\}$ with $g(x)=f(x)$ for $x\not=c$ and $g(c)=c'$. Therefore, any downset of $Q'$ of size $n$ containing $c'$ but not $a$ is also isomorphic to a downset of size $n$ of $Q$ and hence is isomorphic to $A$, $B$, or $C$. Therefore $D$ must include both $a$ and $c'$. 

    Consider the possible structures of $A\setminus X$ and $C\setminus X$ as described in \cref{lem only bits}.  If $A\setminus X$ contains a unique maximal element then this maximal element is $a$ and $D$ must contain all of $A\setminus X$.  If $A\setminus X$ does not contain a unique maximal element then all its maximal elements are at height $h$ and all are indistinguishable by the hypotheses of the statement.  If any one of them is not included in $D$ then by indistinguishability we have the same poset as if $a$ was not in $D$ but this other element was, but that poset, as noted above, must be a copy of $A$, $B$, or $C$ which is a contradiction.  Therefore in this case as well $D$ must contain all of $A\setminus X$.  Using the indistinguishability of the maximal elements we likewise obtain that $D$ must contain all of $C'\setminus X$.  

    However, we need to remove $m-1$ elements from $Q'$ to obtain $D$ so there must be at least $m-1$ elements of $Q'$ which are not in either $A\setminus X$ or in $C'\setminus X$, hence there are at least $m-1$ elements in $X$ which is what we wanted to prove.

\end{proof}

\begin{proof}[Proof of theorem \ref{thm k=3}]
    Continue with notation and set up as earlier in this section.  
    
    We will show that $X$ has size at least $m-3$.  This will give us what we want because $|X| \geq m-3$ gives $3 \geq m -|X|$ and so $|A\cup C| = 2m+|X| = \frac{3}{2}m + \frac{1}{2}m + \frac{3}{2}|X| - \frac{1}{2}|X| = \frac{3}{2}m +  \frac{3}{2}|X| + \frac{1}{2}(m-|X|) \leq \frac{3}{2}m + \frac{3}{2}|X| + \frac{3}{2} = \frac{3}{2}(m+|X|+1) = \frac{3}{2}(n+1)$ which gives the desired result.

    As in \cref{lem max elt structure} consider $A\cup C$ with a maximal element of $A\setminus X$ and a maximal element of $C\setminus X$ removed.  Call this poset $Y$.  By \cref{lem max elt structure} we know that $Y$ has all maximal elements of the same height, $h$, and all elements covered by at least one of these maximal elements is covered by at least $m-2$ of them.  Let $i$ be the number of maximal elements of $Y$ and note that $i\geq 2m-4$.

    We know that all ways of removing $m-2$ of the maximal elements of $Y$ give a copy of $B$, or to say it another way, choosing $i-m+2$ maximal elements of $Y$ along with all the non-maximal elements of $Y$ gives a copy of $B$. 

    For each element $b$ of $B$ of height less than $h$, we can ask how many of the height $h$ elements of $B$ cover this $b$.  Let $M$ be the multiset of these counts.  Note that $M$ is an isomorphism invariant of $B$.  

    Let $y$ be an element of $Y$ that is covered by at least one, hence at least $m-2$ of the maximal elements of $Y$.  Let $j$ be the number of maximal elements of $Y$ covering $y$.  By choosing as few as possible or as many as possible of the elements covering $y$ when making a copy of $B$ as described above we get at most $\min\{j,i-m+2\}$ elements covering $y$ and at least $j-m+2$ (which is nonnegative by \cref{lem max elt structure}) elements covering $y$ in one of these copies of $B$.  Furthermore, all intermediate values are possible by making suitable choices of maximal elements.

    Therefore, the number of different values that the number of covers of $y$ in one of these copies of $B$ can take is
    \begin{align*}
        & \min\{j, i-m+2\} - (j-m+2) + 1 \\
        & = \begin{cases}
            j-(j-m+2)+1 & \text{if }m-2 \leq j \leq i-m+2 \\
            i-m+2-(j-m+2)+1 & \text{if }j > i-m+2
        \end{cases} \\
        & = \begin{cases}
            m-1 & \text{if }m-2 \leq j \leq i-m+2 \\
            i-j+1 & \text{if }j > i-m+2.
        \end{cases} 
    \end{align*}
    If $j \leq i-m+2$ then since all these different values appear in $M$, then $M$ has at least $m-1$ distinct elements and hence $Y$ also has at least $m-1$ elements of height less than $h$.  By \cref{lem only bits}, $A\setminus X$ and $C\setminus X$ can each have at most one element that is not maximal, \textit{i.e.} at most one element not of height $h$. Therefore at most two of the elements of $Y$ with height less than $h$ are contained in $(A\setminus X)\cup (C\setminus X)$, so $X$ contains at least $m-3$ elements and so we are done if any choice of $y$ has $j \leq i-m+2$. 

    Now suppose $j>i-m+2$ and suppose that $y$ was chosen so that $j$ was minimized and that no different choice of maximal element of $A\setminus X$ and $C\setminus X$ to remove in the creation of $Y$ would permit a $y$ with a smaller $j$.  If $j=i$ then by the suppositions of the previous sentence the hypotheses of \cref{lem all covers all} hold and so we have $|X|\geq m-1$ which gives the theorem.
    
   Finally, suppose that $j<i$. Pick a specific copy of $B$ so that the number of covers of $y$ in this $B$ is at the minimum value of $j-m+2$ (\textit{i.e.} choose a copy of $B$ made by removing $m-2$ elements from above $y$). Note that this implies that all the maximal elements of $Y$ not in this copy of $B$ cover $y$.  Additionally, since $j<i$, there is at least one height $h$ element of this copy of $B$ which does not cover $y$.

    Pick two maximal elements $x$ and $z$ of $Y$ with $x$ not in this copy of $B$ (so $x$ covers $y$ in $Y$)  and $z$ in this copy of $B$ but not a cover of $y$.  Adding $x$ and removing $z$ from $B$, we get another copy of $B$, call it $B'$.  $y$ is covered by more than $j-m+2$ elements of height $h$ in $B'$, but $M$ is an isomorphism invariant so some other element $y'$ must be covered by exactly $j-m+2$ elements of height $h$ in $B'$.  As before, every element of height $h$ in $Y$ but not in $B'$ must cover $y'$.  Continuing, pick $x'$ and $z'$ maximal elements of $Y$, with $x'$ not in $B\cup B'$ and $z'$ in $B'$ with the property that $y'$ is not covered by $z'$.  Note that when we remove $z'$ and add in $x'$, the number of height $h$ elements covering $y$ either stays the same or goes up since all of the maximal elements in $Y$ that were not in $B$, including $x'$, cover $y$. 

    Continuing likewise, at the $t+1$th step we have $B, B', \ldots B^{(t)}$, $y, y', \ldots, y^{(t)}$.  We can pick $x^{(t)}$ and $z^{(t)}$ maximal elements of $Y$ so that $x^{(t)}$ is not in $B\cup B'\cup \cdots B^{(t)}$ and hence covers $y, y', \ldots, y^{(t)}$ (since each is covered by all maximal elements of $Y$ which are not in their respective copy of $B$), and we can pick $z^{(t)}$ in $B^{(t)}$ but not covering $y^{(t)}$.  Then in $B^{(t+1)} = B^{(t)} \cup \{x^{(t)}\} \setminus \{z^{(t)}\}$  all of $y, y', \ldots, y^{(t)}$ are covered by more than $j-m+2$ elements and hence there must be some new $y^{(t+1)}$ which is covered by $j-m+2$ elements in $B^{(t)}$.

    This process ends when no maximal elements of $Y$ not in $B\cup B'\cup \cdots$ remain to be chosen, that is, it ends after choosing $x^{(m-3)}$ from which we build $B^{(m-2)}$ and obtain $y^{(m-2)}$.  All the $y^{(t)}$ generated in this way are distinct and so $Y$ has at least $m-1$ non-maximal elements.  At most two of these are in $(A\setminus X)\cup(C\setminus X)$ and so $|X| \geq m-3$ in this case as well which completes the proof.
\end{proof}

\begin{remark}
Our first observation about a shortest $A-B-\cdots-B-C$ path was that the associated $Q$ must have size at most $2n$ (\cref{lem A cup B general}) and we used this in 
\cref{lemma: improved bound general} to give a slightly improved bound for all $k\geq 3$.  One would hope to be able to similarly use \cref{thm k=3} to show that some minimal witness must be even smaller by capturing three of the $P_i$ in at most $\frac{3}{2}(n+1)$ elements and capturing the rest naively with at most $n(k-3)$ additional elements for an improved bound. Unfortunately, this would require further arguments since some of the lemmas use the fact that downsets other than $A$, $B$, $C$ cannot appear in the $k=3$ case,  though we are hopeful that further arguments could be made.  One could hope for even more, where the additional elements do not need to be treated naively and can be collected into opportunities to use \cref{thm k=3} with an average on the order of $\frac{3}{4}n$ per $P_i$, but making such an argument work seems to require quite a bit more, likely some structural results on exchange graphs.
\end{remark}

\section{Conclusion}

Given $\Gamma_n=\{P_1, \ldots, P_k\}$ with distinct $P_i$ of size $n$, if this set has a witness, that is a $Q$ so that the set of downsets of size $n$ of $Q$ is exactly $\Gamma_n$, then we have shown that if $k=3$ then it has a witness of size at most $\frac{3}{2}(n+1)$ (\cref{thm k=3}).  For $k\geq 3$, we have shown that all the minimal witnesses of $\Gamma_n$ are at most of size $n(k-1)$ (\cref{lemma: improved bound general}). We have also given families of examples showing that there exists no upper bound of the form $n+k+c$, for any constant $c$, for the size of the smallest witness of $\Gamma_n$ (\cref{thm section 3}). This remains true even if we restrict the $P_i$ to have height 2 or order dimension 2 (\cref{thm height 2}).  In fact, this last example shows that even $n+2k+c$ is impossible.

We have done an exhaustive exploration of the $k=3$ case up to witness size 9 and up to witness size 10 for height at most three and have found to that point that there is always a witness of size at most $n+3$ (so $m\leq 3$ in the language of \cref{subsec k=3}).  Furthermore, we suspect that in the setup of \cref{subsec k=3} there are sufficient constraints to force either sufficient symmetry that a smaller witness is possible or $m=3$, in a manner somewhat like \cref{lem all covers all} but stronger, though we were not able to carry the proof through.  

More specifically, we suspect that an argument of the basic form as in \cref{lem all covers all} where one can show that unless $m$ is already very small then one can shrink the length of the central chain of $B$s in the path, at the cost of moving to a new $Q'$ that might not be a downset of $Q$, should be possible with weaker hypotheses than those we needed in \cref{lem all covers all}.  In that lemma we used the very strong hypothesis that all the elements of height $h$ had the same past, but all we needed were some isomorphims of certain subposets so as to guarantee that we hadn't built any problematic new downsets of size $n$.  Some suitable weaker hypotheses should still allow this.  Likewise, the counting arguments in the proof of \cref{thm k=3} ought to be able to be improved by using some design theory arguments.  While we were not able to pull any of these improvements through, we think they are likely possible in a way that would result in a bound on $m$ that was still dependent on $n$ or $|X|$, but improved on \cref{thm k=3}.    Obtaining a constant bound on $m$, potentially $m\leq 3$  likely requires some truly new ideas.

Chaining such an $m\leq 3$ result for $k=3$ 
by adding more $P_i$ using the structure of the exchange graph, one could hope for a general bound of $n+3k-3$ which is consistent with all the evidence we are aware of.

A substantial distance remains between $n+3k-3$ and what little we know, and we hope that this question will prove of interest to others and be resolved.

Another direction that one might take in extending these results is searching for ways to ``bootstrap'' covtree, that is searching for rules that tell you directly whether or not a candidate node is a true node without having to search for a witness. For instance, the connectivity of the exchange graph (\cref{thm diameter}) implies that if $\Gamma_n$ is a node in covtree then one can construct a connected graph whose vertices are the elements of $\Gamma_n$ (meaning unlabelled posets, and allowing for the same poset to appear as a vertex more than once) and two vertices are adjacent if and only if one can be obtained from the other by replacing a single maximal element. This condition allows us to rule out candidate nodes without searching for a witness (though it cannot be used to verify that a candidate node is a true node). In a similar vain, constraints such as we obtained on the structure of witnesses (\cref{lem all B}, \cref{lem only bits} and \cref{lem max elt structure}) might in future be used to constrain the order structure of the elements of $\Gamma_n$ directly.
\bibliographystyle{unsrt}
\bibliography{main,causet_refernces}

\end{document}